\documentclass[11pt]{amsart}
\usepackage{amssymb, graphicx}
\usepackage[latin1]{inputenc}
\usepackage{amsfonts}
\usepackage{latexsym}
\usepackage{amscd}
\usepackage{enumerate}
\RequirePackage{color}

\usepackage[all]{xy}
\RequirePackage{color}

\input xy
\xyoption{all}

\numberwithin{equation}{section}

\newtheorem{lemma}{Lemma}[section]
\newtheorem{corollary}[lemma]{Corollary}
\newtheorem{theorem}[lemma]{Theorem}
\newtheorem{proposition}[lemma]{Proposition}
\newtheorem{remark}[lemma]{Remark}

\newtheorem{definition}[lemma]{Definition}
\newtheorem{definitions}[lemma]{Definitions}
\newtheorem{example}[lemma]{Example}

\newcommand{\MM}{{\mathbb{M}}}
\newcommand{\x}{{\mathbf{x}}}
\newcommand{\y}{{\mathbf{y}}}
\newcommand{\s}{{\mathbf{s}}}
\newcommand{\zt}{{\mathbf{t}}}
\newcommand{\MS}{{\mathbf{S}}}

 \DeclareMathOperator{\soc}{Soc}
 
 \usepackage{multicol}
\usepackage{hyperref}
\hypersetup{ colorlinks,
linkcolor=blue,
filecolor=green,
urlcolor=blue,
citecolor=blue }

\definecolor{turquoise2}{rgb}{0,0.898039,0.933333}
\definecolor{magenta}{rgb}{1,0,1}

\begin{document}


\title[Morita equivalence and Morita invariant properties. Leavitt path algebras]
{Morita equivalence and Morita invariant properties. Applications in the context of Leavitt path algebras}

\author[M. Siles ]{Mercedes Siles Molina}
\address{M. Siles Molina: Departamento de \'Algebra Geometr\'{\i}a y Topolog\'{\i}a, Fa\-cultad de Ciencias, Universidad de M\'alaga, Campus de Teatinos s/n. 29071 M\'alaga.   Spain.}
\email{msilesm@uma.es}

\author[J. F. Solanilla]{Jos\'e F. Solanilla Hern\'andez}
\address{J. F. Solanilla Hern\'andez:  Centro Regional Universitario de Cocl\'e: \lq\lq Dr. 
Bernardo Lombardo\rq\rq,  Universidad de  Panam\'a. Apartado Postal 0229. Penonom\'e, 
Provincia de Cocl\'e. Panam\'a.}
\email{jose.solanilla@up.ac.pa}


\begin{abstract}
In this paper we prove that two idempotent rings are Morita equivalent if every corner of one of them is
isomorphic to a corner of a matrix ring of the other one. We establish the converse (which is not true in general)
for $\sigma$-unital rings having a $\sigma$-unit consisting of von Neumann regular elements. The following
aim is to show that a property is Morita invariant if it is invariant under taking corners and under taking matrices. 
The previous results are used to check the Morita invariance of certain ring properties (being locally left/right artinian/noetherian, being categorically left/right artinian, being an $I_0$-ring and being properly purely infinite)
and certain graph properties in the context of Leavitt path algebras (Condition (L), Condition (K) and cofinality). 
A different proof of the fact that a graph with an infinite emitter does not admit any desingularization is also
given.

\end{abstract}

\maketitle


\section*{introduction}
In the study of the structure of unital rings there are several results as powerful as the Wedderburn-Artin 
Theorem, or even more. One of them are Morita's theorems on the equivalence of module categories, that
can be regarded as a generalization of it.  Important consequences are the result that says that two rings $R$
and $S$ are Morita equivalent if and only if $S$ is isomorphic to a corner $e\MM_n(R)e$ for some full idempotent
$e$ in some matrix ring $\MM_n(R)$ and the following corollary: a ring-theoretic property $P$ is Morita invariant
if and only if whenever a ring $R$ satisfies $P$ so does $eRe$, for any full idempotent $e$ and $\MM_n(R)$ for
any $n\geq 2$.

Our aim in this paper has been to study these results for rings without identity. Concretely, we prove that
if two idempotent rings $R$ and $S$ are Morita equivalent, then every corner of $S$ is isomorphic to a corner
of some matrix ring  $\MM_n(R)$, and similarly for every corner of $R$ (Theorem \ref{equivlocal}). We call this
property ``being local matricial isomorphic". While in the unital context if $R$ is local matricial isomorphic to $S$
then $S$ is local matricial isomorphic to $R$, this is not the case under the absence of the unity (see Example
\ref{Nocierto}). We also show the  converse of Theorem \ref{equivlocal} for $\sigma$-unital rings whose $\sigma$-unit
consists of von Neumann regular elements (Theorem \ref{reciproco}). This result is not true in general (see
Example \ref{Noengeneral}). In order to skip the lack of a unit element in our results we need to use the notion of local algebra at
an element. These algebras will also allow to consider von Neumann regular elements in  Theorem \ref{reciproco},
as explained. This is what we do in Section 1.

Section 2 is devoted to settle the precise machinery to test if a property for rings is Morita invariant. We prove
(Theorem \ref{moritainvariant}) that for idempotent rings a property $P$ is Morita invariant if it is stable under
corners and under taking matrices. This result is also true if we change idempotent rings by any subclass of 
idempotent rings. It is used in the section to show the Morita invariance of properties such as
being locally left/right artinian/noetherian, being categorically left/right artinian, being an $I_0$-ring and being properly purely infinite (Theorems \ref{locartinian}, \ref{catartinian}, \ref{I0Mi} and \ref{ppiMI}).

The last section of the article is developed in the context of Leavitt path algebras. We use the results in Section 2
with the aim of getting the Morita invariance of properties related to the underlying graphs of the Leavitt path algebras, 
and show that Condition (L), Condition (K) and cofinality are Morita invariant properties (Theorems \ref{LandK} and \ref{cofinMI}). Finally, we apply the results in Section 1 to prove, following a different approach, the main result in \cite{AR}: a graph which contains an infinite emitter does not admit any desingularization (Theorem \ref{desing}).


\section{Morita equivalence for rings without a unit element. Local rings at elements}

It is well-known that corners of rings, i.e., subrings of a ring $R$ of the form $eRe$, where $e$ is an idempotent of $R$, play a fundamental role in the theory of Morita equivalence for unital rings. For an arbitrary ring, the lack of a unit element can be overcome by using local rings at elements.

\begin{definition}{\rm 
Let $R$ be a ring and let $a\in R$. The \emph{local ring} of $R$ \emph{at} $a$ is defined as $R_a:= aRa$ with sum inherited from $R$ and product given by: $axa \cdot aya = axaya$.}
\end{definition} 

This associative notion was introduced in \cite{FGGS}. The reader is refered to \cite{GS} for a fuller account on transfer of various properties between rings and their local rings at elements. Notice that if $e$ is an idempotent in a ring $R$, then the local ring of $R$ at $e$ is just the corner $R_e$.
\medskip

The other notion that will be of use in order to establish our main results in this section is that of Morita equivalence for idempotent rings. We recall here some of its main aspects. 
\medskip

Let $R$ and $S$ be two rings, $_R N_S$ and $_S M_R$ two bimodules
and $(-, -)\colon N \times M \to R$, $[  -, -]\colon M \times N \to S$ two
maps. Then the following conditions are equivalent:
\begin{enumerate}
\item[(i)] $\left(\begin{matrix}
 R & N \cr M & S\end{matrix}\right)$
  is a ring with componentwise sum and product given by:
\[
 \left(\begin{matrix}r_1 & n_1 \cr m_1 & s_1\end{matrix}\right)
          \left(\begin{matrix} r_2 & n_2 \cr m_2 & s_2\end{matrix}\right) =
         \left(\begin{matrix}
         r_1r_2 +(n_1, m_2) & r_1n_2 +n_1s_2 \cr
          m_1r_2 +s_1m_2 & [  m_1, n_2] +s_1s_2
          \end{matrix}\right)
\]
\item[(ii)] $[ -, -]$ is $S$-bilinear and $R$-balanced,  $( -, -)$
is $R$-bilinear and $S$-balanced and the following associativity
conditions hold:
\[
(n, m)n^{\prime} = n [m, n^{\prime}] \quad \hbox{and} \quad
            [m, n] m^{\prime} = m (n, m^{\prime})\,,
\]
for all $m$, $m'\in M$ and $n$, $n'\in N$.

That $[ -, -]$ is $S$-bilinear and $R$-balanced and that $( -, -)$ is $R$-bilinear and
$S$-balanced is equivalent to having bimodule maps $\varphi \colon N \otimes_S M \to R$ and  $\psi\colon M \otimes_R N \to S$, given by
\[
\varphi (n \otimes m) = (n, m) \quad \hbox{and} \quad \psi (m \otimes n) = [m, n]
\]
so that the associativity conditions above read
\[
\varphi (n \otimes m) n^\prime= n \psi (m \otimes n^\prime) \quad \hbox{and} \quad
       \psi (m \otimes n) m^\prime = m \varphi (n\otimes m^\prime)\,.
\]
\end{enumerate}

A \emph{Morita context} is a sextuple $(R, S, N, M, \varphi,
\psi)$ satisfying one of the (equivalent) conditions given above. The associated ring (in condition (i))
is called the \emph{Morita ring of the context}. By abuse  of
notation we will write $(R, S, N, M)$ instead  of $(R, S, N, M,
\varphi, \psi)$ and  will identify $R$, $S$, $N$ and $M$ with their natural images in
the Morita ring associated to the context. The Morita context is said to be \emph{surjective} if the maps $\varphi$ and $\psi$ are
both surjective.

In classical Morita theory, it is shown that two rings with
identity $R$ and $S$ are Morita equivalent (i.e., $R$-Mod and
$S$-Mod are equivalent categories) if and only if there exists a
surjective Morita context $(R, S, N, M, \varphi, \psi)$. The
approach to  Morita theory for rings without identity by means of
Morita contexts appears in a number of papers (see~\cite{GarciSimon} and the references therein) in which many
consequences are obtained from the existence of a surjective Morita context
for two rings $R$ and $S$.

For an idempotent ring $R$ we denote by $R$-Mod the full
subcategory of the category of all left $R$-modules whose objects
are the ``unital" nondegenerate modules. Here, a left $R$-module
$M$ is said to be \emph{unital} if $M=RM$, and $M$ is said to be
\emph{nondegenerate} if, for $m\in M$, $Rm=0$ implies $m=0$. Note
that, if $R$ has an identity, then $R$-Mod is the usual category
of left $R$-modules.

It is shown in~\cite[Theorem]{Ky} that, if $R$ and $S$ are
arbitrary rings having a surjective Morita context, then the
categories $R$-Mod and $S$-Mod are equivalent. The converse direction is proved in~\cite[Proposition 2.3]{GarciSimon} for idempotent rings, yielding the theorem that follows. Recall that a ring $R$ is said to be \emph{idempotent} if $R^2=R$.

\begin{theorem} Let $R$ and $S$ be two idempotent rings. Then the categories $R$-Mod and
$S$-Mod are equivalent if and only if there exists a surjective
Morita context $(R, S, N, M)$.
\end{theorem}

Given two idempotent rings $R$ and $S$, we will say that they are
\emph{Morita equivalent} if the categories $R$-Mod and $S$-Mod are equivalent.
In what follows, and in order to ease the notation, we will use juxtaposition instead of the tensor product terminology. 
\medskip

Recall that an element $a$ in a ring $R$ is said to be \emph{von Neumann regular} if there exists $b\in R$ such that
$a=aba$. The ring $R$ will be called a  \emph{von Neumann regular ring} if every element in $R$ is von Neumann regular.
\medskip

\begin{definitions}\label{isomloc}{\rm
Let $R$ and $S$ be two rings. 

We will say that $R$ is {\it  local matricial isomorphic to} $S$ if for every $a$ in $R$
which is von Neumann regular (in $R$) there exist $n=n(a)\in \mathbb{N}$ and $u=u^2\in \mathbb{M}_n(S)$ such that the rings $R_a$ and $\mathbb{M}_n(S)_u$ are isomorphic.
\medskip

We will say that $R$ is {\it  corner matricial isomorphic to} $S$ if for every idempotent 
$e$ in $R$ there exist $n=n(e)\in \mathbb{N}$ and $u=u^2\in \mathbb{M}_n(S)$ such that the rings $R_e$ and $\mathbb{M}_n(S)_u$ are isomorphic.
}
\end{definitions}
\medskip

Clearly being local matricial isomorphic implies being corner matricial isomorphic. The converse will also be true by means of the following easy result.

\begin{lemma}\label{isolocales}
Let $a$ be a von Neumann regular element in a ring $R$, and suppose $b\in R$ such that $aba=a$ and  $bab=b$. Denote by $e$ the idempotent $ab$. Then the algebras $R_a$ and $R_e$ are isomorphic. 
\end{lemma}
\begin {proof}
It is not difficult to see that the map $\varphi: R_a\to R_e$ given by $ara\mapsto arab=ab(ar)ab$ is a ring isomorphism.
\end{proof}

\begin{corollary}\label{localmatricialEScornermatricial}
Two rings $R$ and $S$ are local matricial isomorphic if and only if they are corner matricial isomorphic.
\end{corollary}

On the other hand, the notion of being local matricial isomorphic is not symmetric, in the sense that there exist rings $R$ and $S$ such that $R$ is local matricial isomorphic to $S$ but $S$ is not local matricial isomorphic to $R$. See the example that follows.
\medskip

For an arbitrary ring $R$, let ${\rm FM}(R)$ be the set of infinite (countable) matrices over $R$ such that all their entries are zero except at most a finite number of them, and denote by ${\rm RCFM}(R)$ the infinite (countable) matrices over $R$ such that every row and every column has all their entries equal zero except at most a finite number of them. 

\begin{example}\label{Nocierto} {\rm Let $K$ denote a field.
The ring ${\rm FM}(K)$ is local matricial isomorphic to ${\rm RCFM}(K)$ but ${\rm RCFM}(K)$ is not local matricial isomorphic to ${\rm FM}(K)$.
Indeed, note first that $R:={\rm FM}(K)$ is the socle of $S:={\rm RCFM}(K)$, which is an ideal of $S$, hence for every $a\in R$ (which is von Neumann regular because every element in the socle of a ring is) the algebra $R_a$ coincides with $S_a$. This means that $R$ is local isomorphic to $S$. However, for $1$ the unit element of $S$, the algebra $S_1=S$ is not isomorphic to  ${\mathbb M}_n(K)_x$ for any natural number $n$ and any $x \in {\mathbb M}_n(K)$; the reason is that ${\mathbb M}_n(K)_x$ coincides with its socle but not so the ring $S$.
}
\end{example}

Theorem \ref{equivlocal}  extends the  well-known result that asserts that  given two unital rings $R$ and $S$ which are Morita equivalent then  $R$ is isomorphic to a corner of some ring of matrices of size $n\times n$ over $S$. 
\medskip

\begin{theorem}\label{equivlocal}
Let $R$ and $S$ be two Morita equivalent idempotent rings. Then $R$ is local matricial isomorphic $S$ and $S$ is local matricial isomorphic to $R$. 
\end{theorem}
\begin{proof} By Corollary \ref{localmatricialEScornermatricial} it suffices to show that $R$ and $S$ are corner matricial each other.
\medskip
Let $(R, S, N, M)$ be a surjective Morita context for the rings $R$ and $S$, and let $e$ be an idempotent of $R$. Write $e=\sum_{i=1}^n x_is_iy_i$, with $x_i\in N$,  $s_i\in S$ and $y_i\in M$. If we denote $\x=(x_1, \dots, x_n)$, $\s=diag (s_1, \dots, s_n) \in \MM_n(S)$, $ \y=(y_1, \dots, y_n)^t$ and $\MS=\MM_n(S)$, then
$e= \x\s\y$ and the element $u:=\y\x\s\y\x\s$ is an idempotent in $\MS$ as $u^2=\y\x\s\y\x\s\y\x\s\y\x\s=\y(\x\s\y\x\s\y\x\s\y)\x\s=\y e \x\s=u$.

Now, define:

$$\begin{matrix}
\varphi: & eRe & \to & u\MS u \\
& ere & \mapsto & \y(ere)\x\s
\end{matrix}$$

The map  $\varphi$ is well defined because for every $r\in R$, 
$\y(ere)\x\s= \y\x\s\y \x\s\y r x\s\y \x\s\y\x\s= $ $(\y\x\s\y \x\s)\y r x\s(\y \x\s\y\x\s)=$ $ u(\y r x\s)u\in u\MS u$. It is a ring homomorphism as for $\alpha, \beta \in eRe$, $\varphi(\alpha \beta)= \y\alpha \beta\x\s= \y \alpha e\beta\x\s= \y \alpha \x\s\y \beta \x\s = (\y \alpha \x\s) (\y \beta \x\s) =\varphi (\alpha)\varphi(\beta)$. Injectivity: $\varphi(ere)=0$ implies $\y ere\x\s=0$, therefore   $ere=\x\s\y ere\x\s\y =0$. And the surjectivity follows because for $uz u\in u \MS u$, $uzu=\y\x\s \y\x\s z \y\x\s\y\x\s= \y(\x\s\y\x\s z \y\x\s\y)\x\s=\varphi (e\x\s z \y e).$

Finally, note that if we change the roles of $R$ and $S$ we obtain that $S$ is local matricial isomorphic $R$.
\end{proof}
\medskip

An immediate consequence is the following result, which is very well-known for unital rings.

\begin{corollary}\label{casoCorner}
Let $R$ and $S$ be two idempotent rings which are Morita equivalent. Then for every idempotent $e\in R$ there exist a positive integer $n\in \mathbb{N}$ and an idempotent $u\in \MM_n(S)$ such that
$eRe \cong \MM_n(S)_u$. In particular, if $R$ is a unital ring, then $R\cong \MM_m(S)_v$ for a convenient $m\in \mathbb{N}$ and  an idempotent $v\in \MM_m(S)$.
\end{corollary}
 
 \begin{definitions}{\rm
 Consider a class  $\mathcal C$ of rings and let $P$ be a property of rings.

We will say that property $P$  is \emph{stable by corners}  in $\mathcal C$ if for any ring $R$ in the class $\mathcal C$ we have that $R$ satisfies property $P$ if and only if every corner $R_e$, where $e=e^2\in R$, satisfies this property.

We will say that $P$ is \emph{stable by local algebras at von Neumann regular elements}   in $\mathcal C$ if for every ring $R$  in $\mathcal C$, the ring $R$ satisfies $P$ if and only if every local algebra $R_a$ at a von Neumann regular element $a\in R$ satisfies $P$.

We will say that $P$ is \emph{stable by local algebras at elements}  in $\mathcal C$ if for any ring $R$  belonging to $\mathcal C$ we have that $R$ satisfies property $P$ if and only if every local algebra $R_a$ at an element $a\in R$ satisfies this property. 
}
 \end{definitions}
 
 \begin{lemma}\label{stabilityBYcornerISstabilityBYlocalatvNr}
A property $P$ is stable by corners if and only if it is stable by local algebras at von Neumann regular elements.
 \end{lemma}
 \begin{proof}  Since every local algebra at a von Neumann regular element is isomorphic to a corner the result follows by Lemma \ref{isolocales}.
\end{proof}
 
 Recall that a ring $R$ is said to be \emph{semiprime} if it has no nonzero ideals of zero square.
\medskip

  \begin{example}\label{EjemploZocalo}
 {\rm{ \bf (Socle example.)}
The property ``coincidence with the socle" is stable by local algebras at elements for semiprime rings.  We prove this statement. Suppose first that a semiprime ring coincides with its socle. Then every local algebra at an element (which is also semiprime by  \cite[Proposition 2.1 (i)]{GS}) coincides with its socle; the reason is \cite[Proposition 2.1 (v)]{GS}. This same reference implies that if for a semiprime ring every local algebra at an element coincides with its socle, then the ring itself coincides with its socle.

However, the property ``coincidence with the socle" is not stable by corners for semiprime rings. For an example of this fact, let $L$ be a simple domain without identity (such a ring does exist; see \cite[Lemma of Exercise 12.2]{Lex}). Then $L$ is semiprime, as it is a domain; the socle is zero because if $a$ were a nonzero element in the socle of $L$ then $a=aba$ for some $b\in L$, hence $ax=abax$ for every $x\in L$; applying that $L$ is a domain we have $x=bax$, which implies that $ba$ is the unit element of $L$, a contradiction. Since the only idempotent in $L$ is zero, the corner $0L0$ coincides with its socle, therefore $L$ is a ring such that every corner coincides with its socle but $L$ does not coincide with its socle.

The example in the paragraph before also shows that stable by local algebras at elements does not imply stable by corners.
}
 \end{example}
 \medskip

\begin{example}\label{EjemploIntercambio}
 { \bf (Exchange example.)}
 {\rm
The exchange property  (see \cite{AGS} for the definition)
 is stable by local algebras at elements, as shown in \cite[Theorem 1.4]{AGS}.

For the class of rings generated by idempotents (which is closed by corners), the exchange property is stable  by corners. To show this, suppose first that $R$ is a ring generated by the idempotents it contains, and that every corner of $R$ at an idempotent is an exchange ring; then $R$ is an exchange ring by means of \cite[Theorem 3.2]{AGS}. Now, if $R$ is a ring then every corner of $R$ is exchange, as has been proved in the paragraph before.
}
 \end{example}

\bigskip

Theorem \ref{reciproco} is a converse for Theorem \ref{equivlocal} in the case of $\sigma$-unital rings. In order to prove it we need first to establish some notation and the result that follows.
 \medskip

Recall that a ring $R$ is called \emph{$\sigma$-unital} in case there is a sequence $\{u_n\}_{n\in \mathbb N}$ in $R$ such that 
$R=\cup_{n=1}^\infty u_n R u_n$ and $u_n=u_nu_{n+1}=u_{n+1}u_n$ for all $n\geq 1$. The sequence 
 $\{u_n\}_{n\in \mathbb N}$ is called a \emph{$\sigma$-unit}. 
 \medskip
  
Given two idempotents $e, f$ in a ring $R$, we write $e\leq f$ whenever $eRe \subseteq fRf$. With this notation $u_n\leq u_{n+1}$ for every $n\in \mathbb N$ whenever  $\{u_n\}_{n\in \mathbb N}$ is a $\sigma$-unit for a ring $R$.
\medskip

\begin{lemma}\label{desigualdad} Let $R$ and $S$ be two idempotent rings which are Morita equivalent. Let $e$ and $f$ be two idempotents in $R$ such that $e\leq f$. Then there exist a natural number $n$, two idempotents $u, v \in \MM_n(S)$, and ring isomorphisms  $\varphi_e: R_e\to  \MM_n(S)_u$, $\varphi_f: R_f\to \MM_n(S)_v$ such that $u\leq v$ and the following diagram commute:

$$
\begin {matrix} R_f &\overset{\varphi_f}{\longrightarrow} & \MM_n(S)_v\\
{_{i_e}}\uparrow & & \uparrow{_{i_f}}
\\
R_e & \underset{\varphi_e}{\longrightarrow}& \MM_n(S)_u
\end{matrix}
$$
\end{lemma}
\begin{proof}
Consider a surjective Morita context $(R, S, N, M)$ for the rings $R$ and $S$, and write 
$f=\sum_{i=1}^n x_is_iy_i$, with $x_i\in N$,  $s_i\in S$ and $y_i\in M$. Denote $\x=(x_1, \dots, x_n)$, $\s=diag (s_1, \dots, s_n) \in \MM_n(S)$, $ \y=(y_1, \dots, y_n)^t$ and $\MS=\MM_n(S)$, then
$f= \x\s\y$ and the element $v:=\y\x\s\y\x\s$ is an idempotent in $\MS$. Following the proof of Theorem \ref{equivlocal} we see that the map $\varphi_f: fRf\rightarrow \MM_n(S)_v$ given by $\varphi_f(frf)=\y(frf)\x\s$ defines an isomorphism of rings.
\par
Since $e\in R_f$, $e=fef=\x\s\y e \x\s\y$. Define $\zt=\s\y e \x\s\in \MM_n(S)$; then $e=\x\zt\y$. 
Define $u=\y\x\zt\y\x\zt$. We see that $u$ is an idempotent. Indeed, 
$u^2=\y\x\zt\y\x\zt\y\x\zt\y\x\zt= \y (\x\zt\y)^3\x\zt = \y (e)^3\x\zt=  \y (e)\x\zt = \y (\x\zt\y)\x\zt=u$. Moreover, $u\leq v$ as

$\begin{aligned}
u &= \y\x\zt\y\x\zt =\y\x\s\y e \x\s\y\x\s\y e \x\s = \y\x\s\y \x\s\y e \x\s\y \x\s\y\x\s\y \x\s\y e \x\s\y \x\s 
 \\
&
=\y\x\s\y \x\s\y e \x\s\y \x\s\y\x\s\y \x\s\y  \x\s\y e \x\s\y \x\s\y \x\s
\\
&
= (\y\x\s\y \x\s)\y e \x\s\y \x\s\y\x\s\y \x\s\y  \x\s\y e \x\s(\y \x\s\y \x\s) =  (v)\y e \x\s\y \x\s\y\x\s\y \x\s\y  \x\s\y e \x\s(v) 
\end{aligned}$

This shows $u\leq  v$.  Now, define $\varphi_e: eRe\rightarrow \MM_n(S)_u$  by $\varphi_e(ere)= \y (ere)\x\zt$.
\medskip

Since $e=\x\s\y e \x \s \y = \x\zt\y$ is an idempotent, $e= \x\zt\y\x\zt\y\x\zt\y =  \x\zt(\y\x\zt\y\x\zt)\y = \x\zt u \y $. On the other hand, $u = \y\x\zt \y\x\zt = \y e \x \zt$. Now, resoning as in the proof of Theorem \ref{equivlocal} we prove that   
$\varphi_e$ gives an isomorphism of rings. To finish the proof we see that the diagram in the statement is commutative. To this end, take $ere\in eRe$. Then $\varphi_f(ere)= \y er e\x\s =\y ere^3\x\s=
\y ere  \x\zt\y e\x\s=\y ere  \x \s\y e\x\s  \y e\x\s = \y ere (\x \s\y) e\x(\s  \y e\x\s) = \y ere f e\x\zt = \y ere \x \zt =\varphi_e(ere)$. This shows our claim.
\end{proof}

\medskip

\begin{lemma}\label{escalera}
Let $R$  be a $\sigma$-unital ring and let $\{e_n\}$ be a $\sigma$-unit for $R$, where the $e_n$'s are von Neumann regular elements. 
Then: 
$${\underset {\longrightarrow}{lim}}\ \MM_{2^n} (R_{e_n})\cong  {\underset {\longrightarrow}{lim}}\ {\rm FM}(R)_ {\overline e_n^{2^n}} = {\rm FM}(R),$$
 where 
\newpage
$\begin{matrix}
\hspace{4em} \overbrace{\hspace{5.7em}}^{2^n} 
\end{matrix}
$

$\overline e_n^{2^n}=
\begin{pmatrix}
 \begin{array}{ccc}
   e_n &  &  \\
   & \ddots &  \\
   &  & e_n
 \end{array}
  & \vline\begin{array}{ccc}
    &  &  \\
   & 0 &  \\
   &  & 
 \end{array}\\
 \hline
   \begin{array}{ccc}
    &  &  \\
   & 0 &  \\
   &  & \\
 \end{array}& \vline \begin{array}{ccc}
    &  &  \\
   & 0 &  \\
   &  &
 \end{array}
\end{pmatrix}
\in {\rm FM}(R).$
\end{lemma}
\begin{proof}

By Lemma \ref{isolocales} we may assume that $e_n$ is an idempotent for every $n$.

 Take
$(r_{ij})\in {\rm FM}(R)$; let $n_1\in \mathbb N$ be such that $r_{ij}=0$ for all $i, j \geq n_1$; let $n_2\in \mathbb N$ be such that
$r_{ij}\in R_{e_{n_2}}$ for all $i, j$. Take $n= max\{n_1, n_2\}$. Then $(r_{ij})\in {\rm FM}(R)_{\overline e_n^{2^n}}$.
Note that for 
$$e_n^{2^n}:= \left(
\begin{matrix} e_ n & & \\
& \ddots & \\
& & e_n
\end{matrix}
\right) \in \MM_{2^n}(R)$$
we have $ {\rm FM}(R)_{\overline e_n^{2^n}} \cong \MM_{2^n}(R)_{ e_n^{2^n}} \cong \MM_{2^n}(R_{e_n})$.

Lemma \ref{desigualdad} and the transition monomorphisms  given by:

$$\begin{matrix}
\rho_n: & \MM_{2^n}(R_{e_n}) & \longrightarrow & \MM_{2^{n+1}}(R_{e_{n+1}})\\
& x & \mapsto & \left( \begin{matrix}x & 0 \\ 0 & 0 \end{matrix}\right)
\end{matrix}$$
induce transition monomorphisms from ${\rm FM}(R)_{\overline e_n^{2^n}}$ into 
${\rm FM}(R)_{\overline e_{n+1}^{2^{n+1}}}$.
\end{proof}

\begin{lemma}\label{escalera2} 
Let $R$ and $S$ be two $\sigma$-unital rings such that $R$ is local matricial isomorphic to $S$ and $S$ is local matricial isomorphic to $R$. 
Suppose that $R$ and $S$ have $\sigma$-units  $\{e_n\}$ and $\{f_n\}$, respectively, such that $e_n$  and $f_n$ are von Neumann regular elements,
and let $\{u_n\}$ and
$\{v_n\}$ be families of idempotents in  $\mathbb{M}_{t_n}(R)$ and in $\mathbb{M}_{m_n}(S)$, respectively, 
such that
for every $e_n$ the corner $R_{e_n}$ is isomorphic to $\mathbb{M}_{m_n}(S)_{v_n}$ and 
for every $f_n$ the corner $S_{f_n}$ is isomorphic to $\mathbb{M}_{t_n}(R)_{u_n}$, where $m_n$ and $t_n$ are natural numbers.
Then 
${\rm FM}(R) \cong {\rm FM}(S)$.
\end{lemma}
\begin{proof}
By Lemma \ref{isolocales} we may assume that $e_n$ and $f_n$ are idempotents for every $n$.

For an arbitrary $n$, write $u_n=(u_n^{ij})\in \mathbb{M}_{t_n}(R)$, where $\{u_n^{ij}\}\subseteq R= \bigcup R_{e_m}$. Then there exists $q\in \mathbb N$ such that for every $i, j$ we have $u_n^{ij}\in R_{e_q}\cong{ M_{m_q}(S)}_{v_q}$. Consider 
the element $e_q^{t_n}:= diag (e_q, \dots, e_q)\in \mathbb{M}_{t_n}(R)$ and $v_q^{t_n}:= diag (v_q, \dots, v_q)\in \mathbb{M}_{t_nm_q}(S)$. Then $u_n= (u_n^{ij})\in \mathbb{M}_{t_n}(R_{e_q})= \mathbb{M}_{t_n}(R)_{e_q^{t_n}}$, which is isomorphic to $\mathbb{M}_{t_n}\left(\mathbb{M}_{m_q}(S)_{v_q}\right)= \left(\mathbb{M}_{t_nm_q}(S)\right)_{v_q^{t^n}}$. 
This produces a monomorphism $\varphi_n: \mathbb{M}_{t_n}(R)_{u_n} \to \mathbb{M}_{t_n}(R)_{e_q^{t_n}}$ and consequently a monomorphism from $S_{f_n}$ into $\mathbb{M}_{t_nm_q}(S)_{v_q^{t_n}}$ as follows.

$$ (\dag) \quad S_{f_n} \cong \mathbb{M}_{t_n}(R)_{u_n}\overset{\varphi_n}{\longrightarrow} \mathbb{M}_{t_n}(R)_{e_q^{t_n}}\cong
\mathbb{M}_{t_nm_q}(S)_{v_q^{t_n}}$$

Analogously, for every $n\in \mathbb{N}$ we can find a monomorphism $\psi_n$ giving

$$(\dag\dag)\quad R_{e_r} \cong \mathbb{M}_{l_r}(S)_{v_r}\overset{\psi_r}{\longrightarrow} \mathbb{M}_{l_r}(S)_{e_p^{l_r}}\cong
\mathbb{M}_{l_rs_p}(R)_{u_p^{l_r}}$$

Then 

$$
\begin{aligned}
FM(S)
&
\overset {(1)}{=} {\underset {\longrightarrow}{lim}}\ \mathbb{M}_{2^n} (S_{f_n})\cong
 {\underset {\longrightarrow}{lim}}\ \mathbb{M}_{2^n} \left(\mathbb{M}_{t_n}(R)_{u_n}\right)
\overset {(2)}{=}  {\underset {\longrightarrow}{lim}}\ \mathbb{M}_{2^n} \left({\rm FM}(R)_{\overline{u}_n}\right)
\\
&\overset {(3)}{=}  
{\underset {\longrightarrow}{lim}}\ \mathbb{M}_{2^n}({\rm FM}(R))_{\overline u_n^{2^n}}
\overset {(4)}{=}  {\underset {\longrightarrow}{lim}}\ {\rm FM}(R)_{\overline u_n^{2^n}}
\overset {(1)}{=} {\rm FM}({\rm FM}(R)) = {\rm FM}(R)
\end{aligned}$$

(1) By Lemma \ref{escalera}.  
\medskip

(2) For 
$\overline{u}_n:=
\begin{pmatrix}
  u_n
  & \vline &
   0
 \\
 \hline
0
& \vline &
  0
\end{pmatrix}
\in {\rm FM}(R).$
\medskip

(3) Where $\overline u_n^{2^n}:= \left(
\begin{matrix} \overline u_n & & \\
& \ddots & \\
& & \overline u_n
\end{matrix}
\right) \in \MM_{2^n}({\rm FM}(R))$

(4) Because $\{\overline u_n^{2^n}\}$ is a $\sigma$-unit for ${\rm FM}(R)$. This fact is proved in the following lines. Given $x\in R$, let $r$ be in $\mathbb N$ such that $x\in R_{e_r}$, which can be seen as a subring of $\mathbb M_{l_rs_p}(R)_{u_p^{l_r}}$ by $(\dag\dag)$. Take $n\in \mathbb N$ such that $u_p\leq u_n$ and $l_r\leq 2^n$. Then $\overline u_p^{l_r}\leq \overline u_n^{2^n}$, which shows the claim.
\end{proof}

By adapting the ideas of the Brown-Green-Rieffel Theorem, Ara stated a purely algebraic analogue of this theorem. This was precisely the equivalence among conditions (i) and (iii) in the theorem that follows. Here we include a third equivalent condition (under certain restrictions).

\begin{theorem}\label{reciproco}
Let $R$ and $S$ be two idempotent rings. Consider the following conditions:
\begin{enumerate}[\rm (i)]
\item The rings $R$ and $S$ are Morita equivalent.
\item The rings $R$ and $S$ are  local matricial isomorphic each other.
\item ${\rm FM}(R)\cong {\rm FM}(S)$.
\end{enumerate}

Then:

 {\rm (i)} implies {\rm (ii)}. 

If $R$ and $S$ are $\sigma$-unital rings then  {\rm (i)} and  {\rm (iii)} are equivalent. 

If, moreover,
there are $\sigma$-units  $\{e_n\}$ and $\{f_n\}$, for $R$ and $S$, respectively, such that $e_n$ is von Neumann regular in $R$ and $f_n$ is von Neumann regular in $S$, then the three conditions are equivalent.
\end{theorem}
\begin{proof}
(i) implies (ii) follows by Theorem \ref{equivlocal}  and Lemma \ref{desigualdad}.
\par
(i) is equivalent to (iii) as stated in \cite[Theorem 2.1]{Ara}.
\par
(ii) implies (iii) is Lemma \ref{escalera2}.
\end{proof}

\begin{remark}{\rm 
Note that in the unital case condition (i) is equivalent to say that $R$ is local matricial isomorphic to $S$.  This is not what happens for arbitrary $\sigma$-unital rings, as shown in the example that follows.
}
\end{remark}

\begin{example}\label{Noengeneral}{\rm Let $K$ be a field. Consider the rings $R=FM(K)$, and let $S=RCFM(K)$. Then, $R$ is local matricial isomorphic to $S$ as for every idempotent $e$ in $R$, the ring $eRe$ is isomophic to $eSe$, but $R$ and $S$ are not Morita equivalent as $R=\soc(R)=\soc(S)$, but $S$ does not coincide with its socle, and coincidence with the socle is a Morita invariant property for semiprime idempotent rigs (see \cite[Theorem 2.4]{AMMS2}).
}
\end{example}
\medskip


\section{Applications to Morita invariant properties}

In order to show that certain properties are Morita invariant, a technique that has been used is to show that the properties are stable under taking local algebras at elements and taking matrices. Some examples of this can be found in \cite[Theorem 2.1]{AGS}, where it was shown that for the class of idempotent rings the exchange property is Morita invariant, and in 
\cite[Theorem 2.3]{Ara}, where it was proved that for the class of s-unital rings, being von Neumann regular is a Morita invariant property.

These ideas have been used here in order to obtain a standard method to prove when a property is Morita invariant.

\begin{definition} {\rm
Consider a class  $\mathcal C$ of rings and let $P$ be a property of rings.

The property $P$ is said to be \emph{Morita invariant} (in $\mathcal C$) if whenever a ring of $\mathcal C$  satisfies this property, every Morita equivalent ring in $\mathcal C$ also satisfies $P$.

The property $P$ is said to be \emph{stable under taking matrices} (or \emph{by matrices}) if for every ring $R$ in the class, $R$ satisfies $P$ if and only if every matrix algebra $\MM_n(R)$, for any $n\in \mathbb N$, satisfies property $P$.
}
\end{definition}

\begin{theorem}\label{moritainvariant}
A property $P$ is Morita invariant for idempotent rings if it is stable under taking local algebras at  von Neumann regular elements and under taking matrices. Equivalently, the property $P$ is Morita invariant for idempotent rings if it is stable by corners and under taking matrices.

The same statements are true by changing the class  $\mathcal{I}$ of idempotent rings for any class included in $\mathcal{I}$.
\end{theorem}

\begin{proof}
Let $R$ and $S$ be two idempotent rings which are Morita equivalent and suppose that $S$ satisfies condition $P$. By Theorem \ref{equivlocal} for every von Neumann regular element $a\in R$ there exist $n\in \mathbb N$ and  $b\in \MM_n(S)$ such that
$R_a\cong \MM(S)_b$. Since Property $P$ is stable under local algebras at von Neumann regular elements and under taking matrices, $R_a$ satisfies Property $P$. Apply again that this property is stable under taking local algebras at  von Neumann regular elements to get that $R$ satisfies Property $P$. The equivalence follows from Lemma \ref{stabilityBYcornerISstabilityBYlocalatvNr}. 

The last statement is immediate.
\end{proof}

Recall that a ring $R$ is said to \emph{have enough idempotents} if it contains a set of orthogonal idempotents $\{e_\lambda\}_{\lambda\in \Lambda}$ such that $R=\bigoplus_{\lambda\in \Lambda}Re_\lambda =\bigoplus_{\lambda\in \Lambda}e_\lambda R$. 

The set $\{e_\lambda\}_{\lambda\in \Lambda}$ is called a \emph{complete set of idempotents}. 
Following \cite[Proposition 5.20]{Leandro}, if $R$ has enough idempotents, then $R$ is a ring with a set of local units. (Recall that a set $\mathcal{E}$ of commuting idempotents  in a ring $R$ is called a \emph{set of local units for} $R$ if for every finite subset $X$ of $R$ there exists an idempotent $e\in \mathcal{E}$ such that $X\subseteq eRe$).

Examples of rings with enough idempotents are Leavitt path algebras. The class of rings with enough idempotents is contained in the class of idempotent rings. 
\medskip

In what follows we will use this result to show that some properties are Morita invariant.
\medskip

\emph{Locally artinian/noetherian rings.}
\medskip

We say that a ring $R$ is \emph{locally (left/right) artinian/noetherian} if for every finite subset $X$ of $R$ there exists an idempotent $e\in R$ such that
$X\subseteq eRe$ and $eRe$ is (left/right) artinian/noetherian.

\begin{lemma} \label{matrizde} For idempotent rings, the property of  being locally (left/right) artinian/noe\-therian:
\begin{enumerate}[\rm (i)]
\item Is stable under taking matrices. 
\item Is stable by corners.
\end{enumerate}
\end{lemma}
\begin{proof}
(i). Let $R$ be an idempotent ring and suppose that it is locally (left/right) artinian/noe\-therian. Consider a finite subset $X$ of $\MM_n(R)$, and denote by $Y$ the set of the entries of the elements of $X$. Since $R$ is locally (left/right) artinian/noetherian, there exists an idempotent $e\in R$ such that $Y\subseteq eRe$ and the corner $eRe$ is  (left/right) artinian/noetherian. Denote by $f=diag(e, \dotsc, e)\in \MM_n(R)$. Then $X\subseteq f\MM_n(R) f\cong \MM_n(eRe)$, which is (left/right) artinian/noetherian (use \cite[(1.21)]{L}).

(ii) follows because every corner of a (left/right) artininian/noetherian ring is a (left/right) artininian/noetherian ring  by \cite[(21.13)]{L}.
\end{proof}

\begin{theorem}\label{locartinian} For idempotent rings, being locally (left/right) artininian/noetherian is a Morita invariant property.
\end{theorem}

\begin{proof} Apply Lemma \ref{matrizde} and Theorem \ref{moritainvariant}.
\end{proof}
\medskip

\emph{Categorically artinian rings.}
\medskip

Let $R$ be a ring with local units.
We say that $R$ is \emph{categorically left artinian} in case every finitely generated left $R$-module is left artinian/noetherian. The analogous definition of  \emph{categorically right artinian} is obvious. 

\begin{lemma}\label{equivArt}
Let $R$ be a semiprime ring and let $e$ be an idempontent in $R$. The following are equivalent:
 \begin{enumerate}[\rm (i)]
\item $Re$ is a left/right artinian $R$-module.
\item $eRe$ is a left/right artinian ring.
\end{enumerate}\end{lemma}
\begin{proof}

 (i) $\Rightarrow$ (ii). 
Take nonzero left ideals $eL_ne$ of $eRe$, for $n \in \mathbb N$. Then $ReL_ne$ are nonzero left $R$-modules of $Re$. If moreover $eL_ne \subseteq eL_{n+1}e$ then $ReL_ne \subseteq ReL_{n+1}e$. If $Re$ is left artinian then there exists $m\in \mathbb N$ such that $ReL_me = ReL_{m+r}e$ for every $r\in \mathbb N$, therefore for every $x\in L_{m+r}$ we have $exe=zeye$, where $z\in R$ and $y\in L_m$, and so $exe= ezeye \in eReeL_me\subseteq eL_me$. This shows $eL_me = eL_{m+r}e$ for every $r\in \mathbb N$. If we consider the right side the result can be proved analogously.

(ii) $\Rightarrow$ (i). If $eRe$ is an artinian ring, then $e$ belongs to the socle of $R$. This implies that $Re$ has finite uniform dimension, and so it is an artinian left $R$-module.
 \end{proof}

 \begin{proposition}\label{caracterCatArt} Let $R$ be a semiprime ring with enough idempotents and suppose $R=\bigoplus_{\lambda\in \Lambda}Re_\lambda =\bigoplus_{\lambda\in \Lambda}e_\lambda R$, where $\{e_\lambda\}_{\lambda\in \Lambda}$  is a complete set of idempotents. The following conditions are equivalent:
 \begin{enumerate}[\rm (i)]
\item $R$ is categorically left/right artinian.
\item $Re_\lambda$ is a left/right artinian $R$-module for every $e_\lambda$
\item $e_\lambda Re_\lambda$ is a left/right artinian ring for every $e_\lambda$.
\end{enumerate}
 \end{proposition}
 \begin{proof}
 (i) $\Leftrightarrow$ (ii) was proved in \cite[Proposition 1.2]{AAPS}.
 
 (ii) $\Leftrightarrow$ (iii) follows by Lemma \ref{equivArt}.

 \end{proof}
 
 \begin{lemma}\label{EstcatArt}
 For semiprime rings with enough idempotents, the property of being categorically left/right artinian:
  \begin{enumerate}[\rm (i)]
\item Is stable under taking matrices.
\item Is stable by corners.
\end{enumerate}
 \end{lemma}
 \begin{proof}
Let $R$ be a semiprime ring with enough idempotents and take a complete set of idempotents $\{e_\lambda\}_{\lambda\in \Lambda}$ for $R$. 

(i). Suppose $R$ is categorically left artinian. We see that for any $n\in \mathbb N$ the ring $\mathbb M_n(R)$ is categorically left artinian.  It is immediate to see that $\mathbb M_n(R)$ is a semiprime ring. On the other hand, it is a ring with enough idempotents: define $E_\lambda^i$ as the matrix in $\mathbb M_n(R)$ having $e_\lambda$ in place $ii$. Then $\{E_\lambda^i \ \vert \ \lambda \in \Lambda, \ i \in \{1, \dotsc, n\}\}$ is a complete set of idempotents for $\mathbb M_n(R)$. By Proposition \ref{caracterCatArt} to show that $\mathbb M_n(R)$ is a categoricaly left artinian ring, it is enough to see that for any pair $(\lambda, i) \in \Lambda \times  \{1, \dotsc, n\}$ the ring
$E_\lambda^i \mathbb M_n(R)E_\lambda^i$ is left artinian. But this is trivially true as $E_\lambda^i \mathbb M_n(R)E_\lambda^i\cong e_\lambda R e_\lambda$, which is left artinian because $R$ is categorically left artinian and so Proposition \ref{caracterCatArt} can be applied.

(ii). Suppose first $R$ categorically left artinian and let $f$ be an idempotent in $R$. We have show that $fRf$ is categorically left artinian, equivalently (by Proposition \ref{caracterCatArt} applied to the case of a unital ring), $fRf$ is left artinian. By \cite[Proposition 5.20]{Leandro} the ring $R$ has local units; concretely, there exists a finite subset $\{e_1, \dotsc, e_n\} $ of $\{e_\lambda\}_{\lambda\in \Lambda}$ such that $f\in (\sum_{i=1}^n e_i) R (\sum_{i=1}^n e_i)$. 
Use again Proposition  \ref{caracterCatArt} to get that $e_iRe_i$ is a left artinian ring for every $i$. This means that $e_i$ is an idempotent in the socle of $R$. Since the socle is an ideal, this means that the idempotent $\sum_{i=1}^n e_i$ also belongs to the socle of $R$. This is equivalent to say that the corner $ (\sum_{i=1}^n e_i) R (\sum_{i=1}^n e_i)$ is left artinian. Since every corner of a left artinian ring is left artinian, $f (\sum_{i=1}^n e_i) R (\sum_{i=1}^n e_i) f = fRf$ is a left artinian ring. 

Now, suppose that for every idempotent $e$ in $R$ the ring $eRe$ is categorically left artinian (equivalently, it is left artinian). Then, in particular, for every $e_\lambda$, the corner $e_\lambda R e_\lambda$ is left artinian. By Proposition \ref{caracterCatArt} the ring $R$ is categorically left artinian. This shows (ii).
 \end{proof}
 
 \begin{theorem}\label{catartinian}
 For semiprime rings with enough idempotents, being categorically left/right artininian is a Morita invariant property.
 \end{theorem}
 \begin{proof} Apply Lemma \ref{EstcatArt} and Theorem \ref{moritainvariant}.
\end{proof}
\medskip

\emph{$I_0$ rings.}
\medskip
 
 A ring $R$ is said to be an $I_0$ \emph{ring} if every left ideal not contained in the Jacobson radical of $R$ contains a nonzero idempotent.
 
 \begin{lemma}\label{Icero}
  For rings with local units, the property of being an $I_0$ ring:
  \begin{enumerate}[\rm (i)]
\item Is stable under taking matrices.
\item Is stable by corners.
\end{enumerate}
 \end{lemma}
 \begin{proof}
 (i) follows by  \cite[Proposition 1.8]{Nicholson}.
 
 (ii). Let $R$ be a ring with local units. Suppose first that $R$ is an $I_0$-ring. We have to show that for every idempotent $e$ in $R$ the corner $R_e$ is an $I_0$-ring. To this end, we will apply \cite[Lemma 1.1]{Nicholson} and will prove that  every element $exe$ not belonging to the Jacobson radical of $eRe$ is von Neumann regular. Indeed, for $exe$ such an element, since $J(eRe)=eJ(R)e$ (for $J(\ )$ the Jacobson radical) we have $exe\notin J(R)$; by \cite[Lemma 1.1]{Nicholson} $exe$ is von Neumann regular in $R$ and so in $eRe$.
 
 Now, assume that for every idempotent $e$ in $R$ the ring $eRe$ is an $I_0$-ring and show that $R$ is an $I_0$ ring. Take a nonzero $x\in R$ not contained in the Jacobson radical of $R$. Having $R$ local units implies that there exists an idempotent $f$ in $R$ such that $x\in fRf$. Reasoning as in the paragraph before we get that $x=fxf$ does not belong to the Jacobson radical of  $fRf$; apply that $fRf$ is an $I_0$-ring and \cite[Lemma 1.1]{Nicholson} to obtain that $x$ is von Neumann regular in $fRf$ and hence in $R$. This concludes our proof.
 \end{proof}

 \begin{theorem}\label{I0Mi}
 For  rings with local units, being an $I_0$-ring is a Morita invariant property.
 \end{theorem}
 \begin{proof} Apply Lemma \ref{Icero} and Theorem \ref{moritainvariant}.
\end{proof}
\medskip

\emph{Properly purely infinite rings.}
\medskip

Another interesting property of rings is properly purely infiniteness, notion introduced in \cite{AGPS}. 
Recall that a ring $R$ is \emph{properly purely infinite} if every nonzero element is properly infinite (see \cite{AGPS} for the definitions and for an account on the definitions and results concerning these properties).

To finish this section we show that it is a Morita invariant property for rings with local units. As before, for our proof we rely on the facts that being properly purely infinite is invariant under taking matrices and by corners. These results contrast with the analogues for purely infinite rings as, in general, matrix rings over purely infinite rings need not be purely infinite, since otherwise pure infiniteness and strong pure infiniteness would be the same (recall \cite[Lemma 3.4 and Remark 3.5]{AGPS}). However, it was proved in  \cite[Theorem 5.15]{AGPS} that the property of being purely infinite and exchange is Morita invariant.

 \begin{lemma}\label{ppi}
For rings with local units, the property of being properly purely infinite:
\begin{enumerate}[\rm (i)]
\item Is stable under taking matrices.
\item Is stable by corners.
\end{enumerate}
 \end{lemma}
 \begin{proof}
 (i) follows by  \cite[Proposition 5.4]{AGPS}.
 
 (ii). Let $R$ be a ring with local units. If it is properly purely infinite then every corner is properly purely infinite, as follows from \cite[Proposition 5.2]{AGPS}. Now, assume that every corner of $R$ is properly purely infinite and show that $R$ is also properly purely infinite. Take $x$ in $R$ and let $e$ be an idempotent such that $x\in eRe$. By the hypothesis, $eRe$ is properly purely infinite hence, by definition, every element of $eRe$ (in particular $x$) is properly infinite. This shows the claim again by the definition of properly purely infinite ring.
\end{proof}

 \begin{theorem}\label{ppiMI}
 For  rings with local units, being properly purely infinite is a Morita invariant property.
 \end{theorem}
 \begin{proof} Apply Lemma \ref{ppi} and Theorem \ref{moritainvariant}.
\end{proof}
\medskip


\section{Applications in the context of Leavitt path algebras}

In this section we will show that some properties related to the underlying graphs remain invariant by Morita equivalencies among Leavitt path algebras. We start with the essentials on Leavitt path algebras.
\medskip

A \emph{directed graph} is a 4-tuple $E=(E^0, E^1, r_E, s_E)$ consisting of two disjoint sets $E^0$, $E^1$ and two maps
$r_E, s_E: E^1 \to E^0$. The elements of $E^0$ are called the \emph{vertices} of $E$ and the elements of $E^1$ the \emph{edges} of $E$ while for
$e\in E^1$, $r_E(e)$ and $s_E(e)$ are called the \emph{range} and the \emph{source} of $e$, respectively. If there is no confusion with respect to the graph we are considering, we simply write $r(e)$ and $s(e)$.

Given a (directed) graph $E$ and a field $K$, the {\it path $K$-algebra} of $E$,
denoted by $KE$ is defined as the free associative $K$-algebra generated by the
set of paths of $E$ with relations:
\begin{enumerate}
\item[(V)] $vw= \delta_{v,w}v$ for all $v,w\in E^0$.
\item [(E1)] $s(e)e=er(e)=e$ for all $e\in E^1$.
\end{enumerate}

 If
$s^{-1}(v)$ is a finite set for every $v\in E^0$, then the graph is called \emph{row-finite}. If
$E^0$ is finite and $E$ is row-finite, { then} $E^1$ must necessarily be finite as well; in this case we
say simply that $E$ is \emph{finite}.

A vertex which emits no edges is called a \emph{sink}. 
A vertex $v$ is called an \emph{infinite emitter} if $s^{-1}(v)$ is an infinite set, and a \emph{regular vertex} otherwise. 
The  set of infinite emitters {will} be denoted by $ E_{inf}^0$ while ${\rm Reg}(E)$ will denote the set of regular vertices.

The  {\it extended graph of} $E$ is defined as the new graph $\widehat{E}=(E^0,E^1\cup (E^1)^*, r_{\widehat{E}}, s_{\widehat{E}}),$ where
$(E^1)^*=\{e_i^* \ | \ e_i\in  E^1\}$ and the functions $r_{\widehat{E}}$ and $s_{\widehat{E}}$ are defined as 
$${r_{\widehat{E}}}_{|_{E^1}}=r,\ {s_{\widehat{E}}}_{|_{E^1}}=s,\
r_{\widehat{E}}(e_i^*)=s(e_i), \hbox{ and }  s_{\widehat{E}}(e_i^*)=r(e_i).$$

\noindent
The elements of $E^1$ will be called \emph{real edges}, while for $e\in E^1$ we will call $e^\ast$ a
\emph{ghost edge}.    
\medskip

The {\it Leavitt path algebra of} $E$ {\it with coefficients in} $K$, denoted $L_K(E)$, is the quotient of the path algebra $K\widehat{E}$ by the ideal of $K\widehat{E}$ generated by the relations:

\begin{enumerate}    
\item[(CK1)] $e^*e'=\delta _{e,e'}r(e) \ \mbox{ for all } e,e'\in E^1$.
\item[(CK2)] $v=\sum _{\{ e\in E^1\mid s(e)=v \}}ee^* \ \ \mbox{ for every}\ \ v\in  {\rm Reg}(E).$
\end{enumerate}

Observe that in $K\widehat{E}$ the relations (V) and (E1) remain valid and that the following is also satisfied:

\begin{enumerate}
\item [(E2)] $r(e)e^*=e^*s(e)=e^*$ for all $e\in E^1$.
\end{enumerate}
\medskip

 Note that if $E$ is a finite graph, then
$L_{K}(E)$ is unital with $\sum _{v\in E^0} v=1_{L_{K}(E)}$; otherwise, $L_{K}(E)$
is a ring with a set of local units consisting of sums of distinct vertices and that
since every Leavitt path algebra $L_{K}(E)$ has
local units, it is the directed union of its corners.

A \emph{path} $\mu$ in a graph $E$ is a finite sequence of edges $\mu=e_1\dots e_n$
such that $r(e_i)=s(e_{i+1})$ for $i=1,\dots,n-1$. In this case, $s(\mu):=s(e_1)$ and $r(\mu):=r(e_n)$ are the
\emph{source} and \emph{range} of $\mu$, respectively, and $n$ is the \emph{length} of $\mu$. We also say that
$\mu$ is \emph{a path from $s(e_1)$ to $r(e_n)$} and denote by $\mu^0$ the set of its vertices, i.e.,
$\mu^0:=\{s(e_1),r(e_1),\dots,r(e_n)\}$. By $\mu^1$ we denote the set of edges appearing in $\mu$, i.e., $\mu^1:=\{e_1,\dots, e_n\}$.

We view the elements of $E^{0}$ as paths of length $0$. The set of all paths of a graph $E$ is denoted by ${\rm Path}(E)$.
The Leavitt path algebra $L_{K}(E)$ is a
$\mathbb{Z}$-graded $K$-algebra, spanned as a $K$-vector space by
$\{\alpha\beta^{\ast } \ \vert \ \alpha, \beta \in {\rm Path}(E)\}$. In particular, for each $n\in\mathbb{Z}$,
the degree $n$ component $L_{K}(E)_{n}$ is spanned by the set 
$\{\alpha \beta^{\ast }\ \vert \  \alpha, \beta \in {\rm Path}(E)\ \hbox{and}\  \mathrm{length}(\alpha)-\mathrm{length}(\beta)=n\}$. \smallskip

If $\mu$ is a path in $E$, and if $v=s(\mu)=r(\mu)$, then $\mu$ is called a \emph{closed path based
at $v$}. If $s(\mu)=r(\mu)$
and $s(e_i)\neq s(e_j)$ for every $i\neq j$, then $\mu$ is called a \emph{cycle}. 
 A \emph{closed simple path based at a vertex $v$} is a path $\mu = e_1\cdots e_t$ such that $s(\mu)=r(\mu)=v$ and $s(e_i)\neq v$ for all $2\leq i \leq t$.
For $\mu = e_1 \dots e_n\in {\rm Path}(E)$ we write $\mu^*$ for the element $e_n^* \dots e_1^*$ of $L_{K}(E)$.

An edge $e$ is an {\it exit} for a path $\mu = e_1 \dots e_n$ if there exists $i\in \{1, \dots, n\}$ such that
$s(e)=s(e_i)$ and $e\neq e_i$. We say that $E$ satisfies \emph{Condition} (L) if every cycle in
$E$ has an exit and we say that $E$ satisfies  \emph{Condition} (K)  if for each vertex $v$ on a closed simple 
path there exist at least two distinct closed simple paths based at $v$
\medskip

\begin{definition}{\rm
We will say that a property on graphs is \emph{Morita invariant} if given two Leavitt path algebras $L_K(E)$ and $L_K(F)$
which are Morita equivalent (as idempotent rings) then $E$ satisfies that property if and only if $F$ satisfies the same property.
}
\end{definition}

The result that follows has been proved in \cite[ Proposition 3.1.6]{AAS}.

\begin{proposition}\label{Cond(L)book}
Let $K$ be any field and $E$ any graph. The following are equivalent conditions:
\begin{enumerate}[\rm (i)]
\item E satisfies Condition {\rm (L)}.
\item Every nonzero ideal of the Leavitt path algebra $L_K(E)$ contains a nonzero idempotent.
\item Every nonzero left ideal of the Leavitt path algebra $L_K(E)$ contains a nonzero idempotent.
\end{enumerate}
\end{proposition}

\begin{corollary}\label{LesI0}
A graph $E$ satisfies Condition {\rm (L)} if and only if for any field $K$ the Leavitt path algebra $L_K(E)$ is an $I_0$-ring.
\end{corollary}
\begin{proof}
Use Proposition \ref{Cond(L)book}  and \cite[Lemma 1.1]{Nicholson}.
\end{proof}
\begin{theorem}\label{LandK}
 Conditions {\rm(K)} and {\rm(L)} are Morita invariant.
 \end{theorem}
\begin{proof}
Concerning Condition (K), it has been proved that a graph $E$ satisfies Condition (K) if and only if the Leavitt path algebra $L_K(E)$ is an exchange ring (see \cite[Theorem 4.5]{APS1} for the row-finite case and \cite[Theorem 4.2]{G} for the arbitrary case). Since the exchange property is Morita invariant for idempotent rings (see  \cite[Theorem 2.1]{AGS}), we get that Condition (K) is Morita invariant.

That Condition (L) is Morita invariant  follows by Corollary \ref{LesI0} and Theorem \ref{I0Mi}. 
\end{proof}
\medskip

A graph $E$ is said to be \emph{cofinal} if there are no more hereditary and saturated subsets in $E$ than $E^0$ and $\emptyset$. 

\begin{remark}\label{cofin}{\rm
This is equivalent to say that for any field $K$ the Leavitt path algebra $L_K(E)$ is graded simple, as was shown in  \cite[Lemma 2.8]{APS1}.
}\end{remark}
\medskip

Recall that given a $G$-graded ring $R=\oplus_{g\in G}R_g$, for $G$ a group, an ideal $I$ is said to be graded if for any $y\in I$, if $y=\sum_{g\in G}y_g$, then $y_g\in I$ for all $g\in G$. In this case,  $I_g$  will denote $I\cap R_g$.

\begin{remark}{\rm
 It was also proved in \cite[Theorem 4.5]{APS1}, in the row-finite case, and in \cite[Theorem 3.8]{G} in general, that all ideals in a Leavitt path algebra $L_K(E)$ are graded if and only if the graph $E$ satisfies Condition (K). This result, jointly with Theorem \ref{LandK} imply that cofinality is a Morita invariant property. We finish this section by showing in a different way that cofinality is Morita invariant.
}\end{remark}

The following result was settled in \cite[Proposition 3.5]{GarciSimon}. Using similar techniques it can be stablished for graded rings. 

\begin{proposition}\label{isoRet}
Let $G$ be an abelian group, and let $R$ and $S$ be two $G$-graded rings which are idempotent and Morita equivalent. Denote by
$\mathcal{L}_{gr}(R)$ the lattice of graded ideals $I$ of $R$ such that $I_gR_hI_k=I_{ghk}$, and similarly for $S$. Then the lattices $\mathcal{L}_{gr}(R)$ and $\mathcal{L}_{gr}(S)$ are isomorphic.
\end{proposition}

\begin{theorem}\label{cofinMI} Cofinality is a Morita invariant property.
\end{theorem}
\begin{proof} Use Proposition \ref{isoRet} and Remark \ref{cofin} and take into account that every Leavitt path algebra is a $\mathbb Z$-graded ring.
\end{proof}

Our last aim in this section will be to prove Theorem \ref{desing}, which is the main result in \cite{AR}. Here we follow a different approach. We start by recalling the notion of desingularization.
\medskip

If $v_0$ is a sink in $E$, then by \emph{adding a tail at $v_0$} we mean attaching a graph of the form
$$\xymatrix{ {\bullet}^{v_0} \ar[r] & {\bullet}^{v_1} \ar[r] & {\bullet}^{v_2} \ar[r] &
 {\bullet}^{v_3} \ar@{.>}[r] & }$$
to $E$ at $v_0$. If $v_0$ is an infinite emitter  in $E$, then by \emph{adding a tail at $v_0$} we mean
performing the following process: we first list the edges $e_1, e_2, e_3, \ldots$ of $s^{-1}(v_0)$, then we
add a tail to $E$ at $v_0$ of the following form
$$\xymatrix{ {\bullet}^{v_0} \ar[r]^{f_1} & {\bullet}^{v_1} \ar[r]^{f_2} & {\bullet}^{v_2} \ar[r]^{f_3} &
{\bullet}^{v_3} \ar@{.>}[r] & }$$ We remove the edges in $s^{-1}(v_0)$, and for every $e_j \in s^{-1}(v_0)$
we draw an edge $g_j$ from $v_{j-1}$ to $r(e_j)$.

If $E$ is a directed graph, then a \emph{desingularization} of $E$ is a graph $F$ formed by adding a tail to
every sink and every infinite emitter of $E$ in the fashion above. Several basic examples of desingularized
graphs can be found in \cite[Examples 5.1, 5.2 and 5.3]{AA3}.

\begin{remark}\label{desME}
{\rm
If $F$ is a desingularization of an arbitrary graph $E$ then the Leavitt path algebras $L_K(E)$ and $L_K(F)$ are
Morita equivalent. This was shown in  \cite[Theorem 5.2]{AA3} for countable graphs and in  \cite[Lemma 6.7]{T}
for arbitrary graphs.
}
\end{remark}

\begin{theorem}\label{desing}
Let $E$ be a graph that contains an uncountable emitter. Then $E$ does not admit any desingularization.
\end{theorem}
\begin{proof}
 Suppose that $F$ is a desingularization of a graph $E$ and assume that there is an infinite emitter $u \in E^0$ which emits an uncountable amount of edges. Let $s^{-1}(u) =\{e_\alpha\}$, which is an uncountable set. This implies that the set $X:=\{e_\alpha e_\alpha^*\}$, consisting of $K$-linearly independent elements, is also uncountable. 

Consider the corner $L_K(E)_u$. By Remark \ref{desME} and Theorem \ref{equivlocal} there exist a natural number $n\in \mathbb N$ and an idempotent $b\in \MM_n(L_K(E))$ such that the algebras $uL_K(E)u$ and $\MM_n(L_K(F))_b$ are isomorphic. In particular, they have the same dimension as $K$-vector spaces. 

Now, given $b=(b_{ij})$, let $\alpha=v_1+ \dots +v_n$, with $v_i\in E^0$, be such that $b_{ij}\in \alpha L_K(F) \alpha$ and denote by $a=diag(\alpha, \dotsc, \alpha)\in \MM_n(L_K(E))$. Then 
$\MM_n (L_K(F))_b= b\MM_n (L_K(F))b = ba\MM_n (L_K(F))ab = b\MM_n (\alpha L_K(F)\alpha)b$. By \cite[Corollary 8]{AR}, the dimension of $v_iL_K(E)v_j$ is at most countable, for any $i, j \in \{1, \dots, n\}$, hence the dimension of $\alpha L_K(F)\alpha$ is at most countable and so is the dimension of $b\MM_n (\alpha L_K(F)\alpha)b$. But this algebra is isomorphic to $uL_K(E)u$, which contains the linearly independent and uncountable set $X$. This is a contradiction and therefore our result has been proved.
\end{proof}

\section*{Acknowledgments}
The authors have been partially supported by the Spanish MEC and Fondos FEDER through project MTM2010-15223,  by the Junta de Andaluc\'{\i}a and Fondos FEDER, jointly, through projects FQM-336 and FQM-3737. The second author has been partially supported also by the  Programa de Becas para Estudios Doctorales y Postdoctorales SENACYT-IFARHU, contrato no. 270-2008-407,  Gobierno de Panam\'a and by the University of Panam\'a. This work was done
during a research stay of the second author at the University of M\'alaga. He would like to thank the host center for its hospitality and support.

\bibliographystyle{amsplain}

\end{document}